\documentclass[12pt]{amsart}
\usepackage{amsfonts,amscd,latexsym,amsmath,amssymb,enumerate,verbatim,amsthm,epsfig,alltt}
\theoremstyle{plain}
\newtheorem{lemma}{Lemma}[section]
\newtheorem{prop}[lemma]{Proposition}
\newtheorem{thm}[lemma]{Theorem}
\newtheorem{fact}[lemma]{Fact}
\newtheorem{lem}[lemma]{Lemma}
\newtheorem{cor}[lemma]{Corollary}
\newtheorem{question}[lemma]{Question}

\theoremstyle{definition}
\newtheorem{defin}[section]{Definition}

\begin{document}

\title{Base tree property}

\author{Bohuslav Balcar}
\address{The Center for Theoretical Study, Charles University in Prague, 
Jilsk\' a 1, Prague, Czech republic}
\email{balcar@cts.cuni.cz}

\author{Michal Doucha}
\address{Institute of mathematics, Academy of Sciences of the Czech republic, 
\v Zitn\'a 25, Prague, Czech republic}
\email{m.doucha@post.cz}
\thanks{The research of the second author was partially supported by NSF grant DMS 0801114 and grant IAA100190902 of Grant Agency of the Academy of Sciences of the Czech
Republic}

\author{Michael Hru\v{s}\'{a}k}
\address{Instituto de Matem\'aticas, UNAM, Apartado Postal 61-3, Xangari, 58089, Morelia, Michoac\'an, M\'exico.}
\email{michael@matmor.unam.mx}
\thanks{The research of the third author was partially supported by PAPIIT grant IN102311 and
CONACYT grant 177758}


\begin{abstract} Building on previous work from \cite{BPS} we investigate $\sigma$-closed partial orders of size continuum. We provide both an internal and external characterization of such partial orders by showing that (1) every  $\sigma$-closed partial order of size continuum has a base tree and that (2) $\sigma$-closed forcing notions of density $\mathfrak c$ correspond exactly to regular suborders of the
collapsing algebra $Coll(\omega_1, 2^\omega)$.

We further study some naturally ocurring examples of such partial orders.

\end{abstract}

\maketitle
\section*{Introduction}

A partially ordered set  $(P,\leq)$ is \emph{$\sigma$-closed} if every countable decreasing sequence of elements of $P$ has a lower bound. In this note we study
$\sigma$-closed partial orders of size continuum. Orders of this type naturally arise in combinatorial and descriptive set-theory, topology and analysis. 

An essential example is the \emph{collapsing algebra} $Coll(\omega_1, 2^\omega)$, i.e. the completion, in the sense of Boolean algebra, of the complete binary tree of height $\omega_1$. This forcing notion has several
presentations: 
\begin{itemize}
\item $(\mathrm{Fn}(\omega _1,\{0,1\},\omega _1),\supseteq)$ - ordering for adding a new subset of $\omega _1$, 
\item $(\mathrm{Fn}(\omega _1,\mathbb{R},\omega _1),\supseteq)$ - ordering for the consistency of the continuum hypothesis, 
\item $(\mathrm{Fn}(2^{\omega},\{0,1\},\omega _1),\supseteq)$ - ordering for adding $\mathfrak{c}$-many subsets of $\omega _1$, 
\item the natural ordering for adding a $\diamond$-sequence, 
\item Jech's forcing for adding a Suslin tree by countable conditions. 
\end{itemize}

All these orderings are forcing equivalent, in fact, they have isomorphic base trees (see Theorem \ref{BTtheorem} for the term base tree). We refer to \cite{Ku} for definitions of these orderings.

Consider now the set $[\omega]^\omega$ of all infinite sets of natural numbers ordered by inclusion. This order is not $\sigma$-closed, but
it is also not \emph{separative}.\footnote{Recall that a partial order $P$ is separative if whenever $p,q$ are elements of $P$ such that $p\not\leq q$, there is an $r\in P$ such that $r\leq p$ and $r\perp q$.} The separative quotient of 
$([\omega]^\omega, \subseteq)$ are the positive elements in the Boolean algebra
$\mathcal P(\omega)/fin$. In \cite{BPS} the surprising fact that also 
$\mathcal P(\omega)/fin$ has a base tree was established. It was then studied in \cite{Dord}, \cite{Dow2}, \cite{Wil}.

Since then many other naturally occuring examples were studied (\cite{BHH},\cite{BH}) and in each case the methods of \cite{BPS} were used to prove the corresponding Base Tree Theorem. 

In this note we prove this general fact for all partial orders with a dense $\sigma$-closed subset of size continuum. We also identify
the $\sigma$-closed forcings of size continuum as the regular subalgebras of the
collapsing algebra $Coll(\omega_1, 2^\omega)$.

We then present some of the standard examples and review the relevant published results.

We note that similar but somewhat more general notions were studied in \cite{Ko}.

\section{Main results}

The \emph{height} of a partial order $(P,\leq)$, $\mathfrak{h}(P)$ shortly, is the minimal cardinality of a system of open dense subsets of $P$ such that the intersection of the system is not dense. An equivalent definition involves maximal antichains: $\mathfrak{h}(P)$ is equal to the minimal cardinality of a system of maximal antichains from $P$ that do not have a common refinement.
For a Boolean algebra $B$ we define $\mathfrak{h}(B)$ as the height of the ordering $(B\setminus \{0\},\leq)$, where $\leq$ is the canonical ordering on $B$. If $B$ is complete, it coincides with its \emph{distributivity number}. Recall that the distributivity number of $B$ is the least cardinal $\kappa$ such that there exists a matrix $\langle u(\alpha,\beta):\beta\in I_\alpha,\alpha<\kappa\rangle$ of elements from $B$ such that $\bigwedge_{\alpha<\kappa} \bigvee _{\beta\in I_\alpha} u(\alpha,\beta)\neq\bigvee _{f\in \prod_{\alpha<\kappa} I_\alpha}\bigwedge _{\alpha<\kappa} u(\alpha,f(\alpha))$.
We will deal mostly with non-atomic orderings but for completeness we allow atomic orderings in the definition too. Thus, if $(P,\leq)$ is \emph{atomic}, i.e. there is a set of minimal elements such that every other element is above one of them, then we set $\mathfrak{h}(P)=\infty$. Note that for non-atomic orderings height is always a regular cardinal.

The height is a forcing invariant, that means every dense subset of an ordering has the same height. In particular, $\mathfrak{h}(P)=\mathfrak{h}(\mathrm{RO}(P))$.

\begin{fact}
For an ordering $P$, $\mathfrak{h}(P)$ is the minimal cardinal $\kappa$ such that forcing with $P$ adds a new function from $\kappa$ to ordinals. In particular, forcing with $P$ preserves all cardinals less or equal to $\kappa$.
\end{fact}

An ordering $P$ is \emph{homogeneous in $\mathfrak{h}$ (homogeneous in height)} if for every $p\in P$ $\mathfrak{h}(\downarrow p)=\mathfrak{h}(P)$.
The following proposition shows that every partial order can be decomposed into
factors homogeneous in height. For complete Boolean algebras there is a canonical such decomposition.

\begin{prop}
Let $B$ be a complete Boolean algebra. Then $B\cong\prod _{b\in I} B\upharpoonright b$, where $I$ is a partition of unity and $B\upharpoonright b$ is homogeneous in height for every $b\in I$.

Moreover, $\mathfrak{h}(B\upharpoonright a)\neq \mathfrak{h}(B\upharpoonright b)$ if $a\neq b$ for $a,b\in I$.
\end{prop}
\begin{proof} Let $A$ be the set of all atoms, then $B\upharpoonright \bigvee A$ is the first factor homogeneous in the height $\infty$.

Next, we work with an atomless complete algebra $B_0=B\upharpoonright (-\bigvee A)$ ($B\cong B_0\times B\upharpoonright \bigvee A$). Let $(D_{\alpha})_{\alpha <\mathfrak{h}(B_0)}$ be the system of open dense subsets of $B_0$ such that $\bigcap _{\alpha <\mathfrak{h}(B_0)} D_{\alpha}$ is not dense. Let $A_1$ be the subset of elements of $B_0$ witnessing the non-density, i.e. $\downarrow a\cap \bigcap _{\alpha <\mathfrak{h}(B_0)} D_{\alpha}=\emptyset$ for every $a\in A_1$. We claim that for every $a\in A_1$ $B\upharpoonright a$ is homogeneous in the height (with height $\mathfrak{h}(B_0))$. Assume not, then there is some $a\in A_1$ and $b<a$ such that $\mathfrak{h}(B\upharpoonright b)<\mathfrak{h}(B\upharpoonright a)$. Thus, there is a system $(S_{\alpha})_{\alpha <\mathfrak{h}(B\upharpoonright b)}$ of open dense subsets of $B\upharpoonright b$ with a non-dense intersection below $b$. However, if we set $D_{\alpha}=S_{\alpha}\cup B_0\setminus \downarrow b$ then we get a system of open dense subsets in $B_0$ without a dense intersection less than $\mathfrak{h}(B_0)$, that is a contradiction.

We take the join $\bigvee A_1$ of all elements from $A_1$ and the factor $B\upharpoonright \bigvee A_1$ is homogeneous in height. We continue with the remainder $B_1=B_0\upharpoonright (-\bigvee A_1)$ and in the same way get a set $A_2$ of elements witnessing the non-density of the intersection of a system of open dense subsets of size $\mathfrak{h}(B_1)$. It is possible that $\mathfrak{h}(B_1)=\mathfrak{h}(B_0)$. In this case, we join the elements of $A_2$ with the elements of $A_1$. In the opposite case, $B_1\upharpoonright \bigvee A_2$ is a new factor homogeneous in height.

We continue similarly until we treat all elements of $B$. We end up with the desired decomposition.
\end{proof}

\begin{defin}[Base tree property] An ordering $(P,\leq)$ has the \emph{Base Tree Property} (we shall shortly say it has the {\it BT-property}) if it contains a dense subset $D\subseteq P$ with the following three properties:
\begin{itemize}
\item[-] it is atomless; i.e. for every $d\in D$ there are elements $d_1,d_2\in D$ below $d$ such that $d_1\perp d_2$
\item[-] it is $\sigma$-closed
\item[-] $|D|\leq \mathfrak{c}$
\end{itemize}
\end{defin}

It can be easily seen that assuming the Continuum Hypothesis, all partial orders
with the BT-property are forcing equivalent with $Coll(\omega_1, 2^\omega)$ and,
 consequently have a base tree. In fact, the following is true in ZFC.

\begin{thm}[The base tree theorem]\label{BTtheorem} Let $(P,\leq)$ be an ordering homogeneous in height with the BT-property. Then there are $\mathfrak{h}(P)$ maximal antichains $(T_{\alpha})_{\alpha <\mathfrak{h}(P)} \subseteq P$ such that:
\begin{enumerate}[(i)]
\item $(T=\bigcup _{\alpha <\mathfrak{h}(P)} T_{\alpha}, \geq)$ is a tree of height $\mathfrak{h}(P)$, where $T_{\alpha}$ is the $\alpha$-th level of the tree,
\item each $t\in T$ has $\mathfrak{c}$ immediate successors,
\item $T$ is dense in $P$.
\end{enumerate}
$T$ is called a base tree of $P$.\\
\end{thm}
\begin{proof} We need to work with some dense subset guaranteed by the definition of the BT-property rather than with $P$ itself. To avoid introducing new symbols and sets, we assume $P$ itself has the properties.

We use the definition of height. So we can find a system $(A_{\alpha})_{\alpha<\mathfrak{h}(P)}$ of dense open subsets with a non-dense intersection. However, we need to ensure the intersection to be empty. For this, we will work in the completion $\mathrm{RO}(P)$. Suppose $\bigcap _{\alpha <\mathfrak{h}(P)} A_{\alpha}$ is not empty. Consider the element $a=\bigvee (\bigcap _{\alpha <\mathfrak{h}(P)} A_{\alpha})\in \mathrm{RO}(P)$. Since $P$ is homogeneous in height, also $\mathrm{RO}(P)$ is homogeneous in height, and there is a system $(\bar{A} _{\alpha})_{\alpha <\mathfrak{h}(P)}$ of dense open sets of $P$ below $a$ of the same size such that their intersection is not dense below $a$. We replace $\downarrow a\cap A_{\alpha}$ by $\bar{A} _{\alpha}$ (i.e. $(A_{\alpha}\setminus \downarrow a)\cup \bar{A} _{\alpha}$). We get a new system $(A'_{\alpha})_{\alpha<\mathfrak{h}(P)}$ of dense open subsets of $P$ with a non-dense intersection. If this intersection is again non-empty, we again consider the element $a\geq b=\bigvee (\bigcap _{\alpha <\mathfrak{h}(P)} A'_{\alpha})\in \mathrm{RO}(P)$ and continue similarly. We can repeat this procedure until we get the desired  system $(B_{\alpha})_{\alpha <\mathfrak{h}(P)}$ of dense open subsets with an empty intersection.

Next, we extract from each dense open set $B_{\alpha}$ a maximal antichain $C_{\alpha}$. We claim that for every $p\in P$ there is at least one maximal antichain $C_{\alpha}$ and elements $a,b\in C_{\alpha}$ such that $p$ is compatible with both of them: suppose that for some $p\in P$ and for every $\alpha <\mathfrak{h}(P)$ there is only one element $c_{\alpha}$ from $C_{\alpha}$ that is compatible with $p$. However, $p$ is then, in fact, below $c_{\alpha}$ (since if $p\nleq c_{\alpha}$ then there is a $p_0\leq p$ that is disjoint with $c_{\alpha}$ but necessarily compatible with another element of $C_{\alpha}$). This means that $p\in \bigcap _{\alpha <\mathfrak{h}(P)} \downarrow C_{\alpha} \subseteq \bigcap _{\alpha <\mathfrak{h}(P)} B_{\alpha}$ and that is a contradiction with the fact that the intersection is empty.

Before constructing the levels of $T$ we modify the antichains into a system $(D_{\alpha})_{\alpha <\mathfrak{h}(P)}$ where $D_{\beta}$ refines $D_{\alpha}$ if $\alpha <\beta$. This can be easily done if we set $D_{\alpha}$ to be a common refinement of $(C_{\gamma})_{\gamma \leq \alpha}$ and  $(D_{\gamma})_{\gamma < \alpha}$.

The levels of the tree $T$ will be maximal antichains. What we need to take care of is to ensure that $T$ is dense and that every element of $T$ has $\mathfrak{c}$ immediate successors. We begin by showing that for each element $p\in P$ there is an antichain $D_{\alpha}$ with $\mathfrak{c}$-many elements compatible with $p$. There is some $D_{\alpha _0}$ and elements $d_0,d_1\in D_{\alpha _0}$ compatible with $p$, i.e. there are elements $p_0\leq d_0,p_1\leq d_1$ below $p$. Then again there is some $D_{\alpha _1}$ and elements $d_{00},d_{01},d_{10},d_{11}\in D_{\alpha _1}$, the first two compatible with $p_0$, the last two with $p_1$ (note that this is the place where we need the antichains to be refining; since in general there would be some $D_{\beta _1}$ with compatible elements with $p_0$ and some $D_{\beta _2}$ with compatible elements with $p_1$ but in our case we can take $\alpha _1$ to be $\sup \{\beta _1,\beta _2\}$). We again get $p_{\zeta}\leq p$ for each $\zeta \in {}^2\{0,1\}$. We continue until we get an appropriate $p_{\zeta}\leq p$ for each $\zeta \in {}^{<\omega}\{0,1\}$. For every $\xi \in {}^{\omega}\{0,1\}$ we have a descending chain $p\geq p_{\xi \upharpoonright \{0\}}\geq \ldots p_{\xi \upharpoonright n}\geq \ldots$ with a lower bound $p_{\xi}$ (due to $\sigma$-closedness). Moreover, $p_{\xi _1}\perp p_{\xi _2}$ for $\xi _1\neq \xi _2$. Thus we see that there is a maximal antichain of size $\mathfrak{c}$ below $p$; we denote it $\mathcal{A}(p)$. Each such $p_{\xi}$ is compatible with some element $d_{\xi}$ of $D_{\alpha}$ where $\alpha =\sup \{\alpha _n:n\in \omega\}$. And again $\xi _1\neq \xi _2$ implies $d_{\xi _1}\neq d_{\xi _2}$.

Let $P_{\alpha}=\{p\in P:p$ is compatible with $\mathfrak{c}$-many elements of $D_{\alpha}\}$. We see that $P=\bigcup _{\alpha <\mathfrak{h}(P)} P_{\alpha}$. Since $|P_{\alpha}|\leq \mathfrak{c}$ for each $\alpha$ there is an injective mapping $f_{\alpha}:P_{\alpha}\rightarrow {}^{\omega}2$ such that $p_{f_{\alpha}(p)}\leq p$ for every $p \in P_{\alpha}$, where $p_{f_{\alpha}(p)}$ is from the construction above. For each $\alpha$ for which $P_\alpha$ is non-empty we complete the antichain $\{p_{f_{\alpha}(p)}:p\in P_{\alpha}\}$ into a maximal antichain $R_\alpha$.

Now we are ready to start the construction of the base tree. We set $T_0=D_0$ and for each $\alpha +1$ we set $T_{\alpha +1}$ to be the common refinement of $D_{\alpha +1}$, $\mathcal{A}(p)$ for all $p\in T_{\alpha}$ and $R_\alpha$ if it exists. For $\alpha$ limit, $T_{\alpha}$ is a common refinement of $(T_{\gamma})_{\gamma <\alpha}$.

Note that by refining $\mathcal{A}(p)$ for all $p\in T_{\alpha}$ we ensure that each element of the tree has $\mathfrak{c}$-many immediate successors and by refining $R_\alpha$'s that $T$ is dense. This finishes the proof.
\end{proof}

\begin{cor}\label{baselem} For an ordering $(P,\leq)$, the following statements  are equivalent:
\begin{enumerate}[(i)]
\item $P$ has the BT-property,
\item $P$ has a dense subset with the BT-property,
\item Every dense subset of $P$ has the BT-property,
\item $\mathrm{RO}(P)$ has the BT-property.

\end{enumerate}
\end{cor}

\begin{proof}
Note that (i)$\Rightarrow$(ii) and (iii)$\Rightarrow$(iv) follow from the definition. It suffices to prove (ii)$\Rightarrow$(iii), (iv)$\Rightarrow$(i) is then a consequence.

We need to find a dense subset of a given dense subset that is, moreover, $\sigma$-closed and of size $\mathfrak{c}$, atomlessness is clear.

Assuming (ii), we have a base tree $T$, we are given a dense subset $D$ and we show that there is a $\sigma$-closed dense subset $S\subseteq D$.

We make $S$ from maximal antichains. For every $t\in T_0$ we find a maximal antichain $A_t\subseteq D$ below the element $t$. $\bigcup _{t\in T_0} A_t$ is the first maximal antichain $S_0$.

Then for every $s\in S_0$ we find a maximal antichain $M_s\subseteq T$ below $s$. Let $P_1\subseteq T$ be a maximal antichain from $T$ refining $\bigcup _{s\in S_0} A_s$ and $T_1$. Again, for every $p\in P_1$ we find a maximal antichain $A_p\subseteq D$ from $D$, the union $\bigcup _{p\in P_1} A_p$ is $S_1$.

Isolated steps are treated similarly. We need not omit $P_{\alpha}$ to be refining the tree level $T_{\alpha}$. Then we refine it to $S_{\alpha}\subseteq D$.

For a limit $\alpha$ we take a refinement $P_{\alpha}$ of all $P_{\beta}$'s for $\beta <\alpha$ (which is also a refinement of $S_{\beta}$'s) together with $T_{\alpha}$. Then we again refine it to $S_{\alpha}\subseteq D$.

The resulting set $S=\bigcup _{\alpha <\mathfrak{h}(P)} S_{\alpha}$ is dense and $\sigma$-closed. We ensured density by refining all levels of $T$. For $\sigma$-closedness observe that for every countable descending chain $s_0\geq s_1\geq \ldots$ from $S$, where $s_n\in S_{\alpha _n}$, there is an interwined descending chain $p_0\geq p_1\geq \ldots$ such that $p_0\geq s_0\geq p_1\geq s_1\geq \ldots$, where $p_n\in P_{\alpha _n}$. This interwined chain has a lower bound $p$ in $P_{\alpha}$, where $\alpha =\sup \{\alpha _n: n\in \omega \}$, and $p$ has some successor $s\in S$.
\end{proof}

In other words, having a $\sigma$-closed dense set is preserved by forcing equivalence among separative partial orders of size continuum. On the other hand, Zapletal in \cite{Zap} has constructed  a model in which the Continuum Hypothesis holds  and there are two forcing equivalent separative partial orders of size $\aleph_2$ one $\sigma$-closed and the other without a $\sigma$-closed dense set. One has to wonder whether
such a pair exists in ZFC.
\begin{question}
Are there, in ZFC, two separative partial orders which are forcing equivalent such that one is $\sigma$-closed and the other does not have a $\sigma$-closed dense set? Can such partial orders be found as dense subsets of the completion of $\mathrm{Fn}(\omega_1,\mathfrak{c}^+)$?
\end{question}

Finally, using this internal characterization of the partial orders with the BT-property one can easily deduce the following external characterization.

\begin{thm}\emph{}
\begin{enumerate}
\item Let $(P,\leq)$ be an ordering with the BT-property. Then $\mathrm{RO}(P)$ is a regular subalgebra of  $Coll(\omega _1,\mathfrak{c})$.
\item Let $B$ be a complete atomless regular subalgebra of $Coll(\omega _1,\mathfrak{c})$. Then $B$ has the BT-property. 

\end{enumerate}

\end{thm}
\begin{proof}
We prove (1): Let $D\subseteq P$ be its dense subset witnessing the BT-property. Then $D\times \mathrm{Fn}(\omega _1,\{0,1\},\omega _1)$ with induced product ordering clearly has the BT-property, the height is $\omega _1$, thus it is forcing equivalent with the complete Boolean algebra $Coll(\omega _1,\mathfrak{c})$. Note that there is a regular embedding $e:D\rightarrow D\times \mathrm{Fn}(\omega _1,\{0,1\},\omega _1)$ defined as $e(d)=(d,1)$ where $1$ is the biggest element in $\mathrm{Fn}(\omega _1,\{0,1\},\omega _1)$, i.e. the empty set. $e$ is extended onto $\bar{e}:\mathrm{RO}(P)\rightarrow Coll(\omega _1,\mathfrak{c})$ mapping $\mathrm{RO}(P)$ on a regular subalgebra of $Coll(\omega _1,\mathfrak{c})$.\\

Now we prove (2): We use the result of \cite{Ve} that every atomless game-closed complete Boolean algebra of density at most $\mathfrak{c}$  has a $\sigma$-closed dense subset (of size at most $\mathfrak{c}$ in fact). We note that the converse is trivial; see also \cite{Ve}. Since $\mathrm{Fn}(\omega _1,\{0,1\},\omega _1)$ contains a $\sigma$-closed dense subset it is game-closed. Now let $B$ be any complete atomless regular subalgebra, it is still game-closed and of density at most $\mathfrak{c}$. Thus it has a $\sigma$-closed dense subset which is atomless and of size at most $\mathfrak{c}$ thus it witnesses that $B$ has the BT-property.
\end{proof}

\section{Classical examples}
The Boolean algebra $\mathcal P(\omega)/fin$ is a prototype of an ordering with the BT-property. Recall the definitions of the cardinal invariants $\mathfrak{p},\mathfrak{t}$ (\cite{Bl}). It was proved recently by M. Malliaris and S. Shelah (\cite{MaSh}) that $\mathfrak{p}=\mathfrak{t}$. We shall discuss these cardinal invariants on other orderings too.

The second fundamental example is $(\mathrm{Dense}(\mathbb{Q}),\subseteq)$, where $\mathrm{Dense}(\mathbb{Q})$ is a set of all dense subsets in rationals. The situation here is similar with the previous example, it is not separative and the ordering $(\mathrm{Dense}(\mathbb{Q}),\subseteq)$ itself does not satisfy the BT-property. The separative modification is $(\mathrm{Dense}(\mathbb{Q}),\subseteq _\mathrm{nwd})$, where $A\subseteq _\mathrm{nwd} B$ if $A\setminus B$ is nowhere dense in $\mathbb{Q}$, has the BT-property. This ordering is studied in \cite{BHH}.

Let $\mathfrak{p}_\mathbb{Q},\mathfrak{t}_\mathbb{Q},\mathfrak{h}_\mathbb{Q}$ be the cardinal invariants of $(\mathrm{Dense}(\mathbb{Q}),\subseteq _\mathrm{nwd})$ defined in the same way as their counterparts in $([\omega]^{\omega},\subseteq ^*)$. It was proved in \cite{BHH} that $\mathfrak{p}_\mathbb{Q}=\mathfrak{p}$ and $\mathfrak{t}_\mathbb{Q}=\mathfrak{t}$ (thus $\mathfrak{p}_\mathbb{Q}=\mathfrak{t}_\mathbb{Q}=\mathfrak{t}$) whereas $\mathfrak{h}_\mathbb{Q}$ and $\mathfrak{h}$ are incomparable in ZFC, $\mathfrak{h}_\mathbb{Q}<\mathfrak{h}$ and $\mathfrak{h}_\mathbb{Q}>\mathfrak{h}$ are both consistent (see \cite{BHH} and \cite{Br}); and $\mathfrak{h}_\mathbb{Q}=\mathfrak{h}$ too, of course. 

For the third example, let $\mathbb{A} $  be the \emph{Cantor algebra}, i.e. the algebra of all clopen subset of $2^{\omega}$, and consider the countable product $\mathbb{A}^{\omega}$ modulo the ideal $\mathrm{Fin}\subseteq \mathbb{A}^{\omega}$, where $\mathrm{Fin}=\{f\in \mathbb{A}^{\omega}: |\{n:f(n)\neq 0\}|<\omega\}$. It satisfies the BT-property, moreover, $\mathbb{A}^{\omega}/\mathrm{Fin}$ is homogeneous.

$\mathfrak{t}(\mathbb{A}^{\omega}/\mathrm{Fin})=\mathfrak{t}$ and $\mathfrak{h}(\mathbb{A}^{\omega}/\mathrm{Fin})\leq \min\{\mathfrak{h},\mathrm{add}(\mathcal{M})\}$ (\cite{BH}) and it is consistent that $\mathfrak{h}(\mathbb{A}^{\omega}/\mathrm{Fin})<\mathfrak{h}$ (\cite{BH},\cite{Dow}).

\smallskip

For any Boolean algebra $B$ let us consider an infinite product $B^{\omega}$. Let $J$ be an ideal on $\omega$. By $\mathcal{I}_J\subseteq B^{\omega}$ we denote the ideal $\{f\in B^{\omega}: \{n\in \omega :f(n)\neq 0\}\in J\}$. The quotient algebra $B^{\omega}/\mathcal{I}_J$ consists of equivalence classes where $f,g\in B^{\omega}$ are equivalent if $\{n: f(n)\neq g(n)\}\in J$ ($f\bigtriangleup g\in \mathcal{I}_J$ equivalently). We state and prove a simple criterion for when such a product has the BT-property.
\begin{thm}
Let $B$ be a Boolean algebra and $J$ an ideal on $\omega$. Then the reduced product $B^{\omega}/\mathcal{I}_J$ has the BT-property if and only if $B$ contains a dense subset of size $\mathfrak{c}$ and (either $\mathcal{P}(\omega)/J$ is $\sigma$-closed or $J$ is a maximal ideal), and (either $\mathcal{P}(\omega)/J$ or $B$ is atomless).
\end{thm}
\begin{proof}
Since $B^{\omega}/\mathcal{I}_J$ contains a dense subset of size $\mathfrak{c}$ if and only if $B$ contains a dense subset of size less or equal to $\mathfrak{c}$ the requirement on the cardinality is satisfied.

Suppose that $\mathcal{P}(\omega)/J$ is not $\sigma$-closed. Let $(X_n)_{n\in \omega}$ be a descending chain of infinite subsets  of $\omega$ such that the chain $([X_n])_{n\in \omega}$ does not have a lower bound in $\mathcal{P}(\omega)/J$, where $[X_n]$ is the equivalence class containing $X_n$. We define a descending chain $([f_n])_{n\in \omega}\subseteq B^{\omega}/\mathcal{I}_J$ as follows: $f_n(i)=1$ if $i\in X_n$ and $f_n(i)=0$ otherwise (it is the image of the chain $([X_n])_{n\in \omega}$ via the regular embedding of $\mathcal{P}(\omega)/J$ into $B^{\omega}/\mathcal{I}_J$). Suppose that it has a lower bound $[f]$. Then the support of $f$, i.e. the set $\{i: f(i)\neq 0\}$, would determine a lower bound for $([X_n])_{n\in \omega}$.

Next we use the fact mentioned in \cite{BH} that $B^{\omega}/\mathcal{I}_J$ can be written as an iteration $\mathcal{P}(\omega)/J\star B^{\omega}/\dot{\mathcal{U}}$, where $\dot{\mathcal{U}}$ is a name for an ultrafilter added by $\mathcal{P}(\omega)/J$. For $[f]\in B^{\omega}/\mathcal{I}_J$ we define $\Phi([f])=(\{i:f(i)\neq 0\},\dot{[f]})$, where $\dot{[f]}$ is a name for an equivalence class containing $f$ in $B^{\omega}/\dot{\mathcal{U}}$. $\Phi$ is easily verified to be a dense embedding which proves the fact.

Now observe that an ultrapower of any Boolean algebra is $\sigma$-closed. For a countable descending chain we can choose representatives of equivalence classes $(f_n)_{n\in \omega}$ so that support $f_0=\omega$, support $f_1\supseteq$ support $f_2\supseteq$ support $f_3\supseteq \ldots$ and $\bigcap _{n\in \omega}$ support $f_n=\emptyset$ since the ultrafilter is non-principal. Then we set $f(i)=f_n(i)$ if $n$ is the smallest number such that $i\in$ support $f_n\setminus$ support $f_{n+1}$. $f$ clearly determines the lower bound for the chain. Hence, we conclude that $B^{\omega}/\mathcal{I}_J$ is $\sigma$-closed since an iteration of two $\sigma$-closed forcings is.

To check atomlessness, if $\mathcal{P}(\omega)/J$ is atomless then for any $f\in B^{\omega}$, where the support of $f$ is not in $J$, we can always split the support of $f$ into two disjoint infinite sets both outside of $J$, restrict $f$ to these sets and make two disjoint elements of $B^{\omega}/\mathcal{I}_J$ below $[f]$. If $B$ is atomless then we can always find two disjoint successors coordinatewise. Finally, suppose that $\mathcal{P}(\omega)/J$ has an atom $[A]$ and $B$ has an atom $b$. Then $f\in B^\omega$ defined so that $f(n)=b$ for $n\in A$ and $f(n)=0$ for $n\notin A$ determines an atom in $B^{\omega}/\mathcal{I}_J$.
\end{proof}

\section{On classes of ideals ordered by reverse inclusion}
We shall deal with orderings that consist of ideals on $\omega$ of some type ordered by reverse inclusion. We assume that all such ideals extend the ideal of finite subsets of $\omega$. Since every ideal on $\omega$ can be considered as a subset of the Cantor space we can speak about the topological, resp. measure-theoretical characterizations of such ideals.
\subsection{Non-tall ideals}
An ideal $\mathcal{I}$ on $\omega$ is \emph{tall} if for every $X\in [\omega]^{\omega}$ there is infinite $Y\subseteq X$ that belongs to $\mathcal{I}$. Consider the set $\mathfrak{T}$ of all non-tall ideals on $\omega$ ordered by reverse inclusion.

First of all, this ordering is not separative. However, for every $A\in [\omega]^{\omega}$ consider the ideal $I_A$ of all subsets of $\omega$ that have a finite intersection with $A$. $I_A$ is a non-tall ideal and $B\subseteq ^* A$ implies $I_B\supseteq I_A$. Moreover, for every non-tall ideal $\mathcal{I}$ and some infinite set $A$ almost disjoint with every element of $\mathcal{I}$, $I_A\supseteq \mathcal{I}$. Thus we see that $([\omega]^{\omega},\subseteq ^*)$ is isomorphic with a dense subset of $(\mathfrak{T},\supseteq )$ and of its seprative modification showing that the separative modification of $(\mathfrak{T},\supseteq )$ has the BT-property, however it is forcing equivalent to $([\omega]^{\omega},\subseteq ^*)$.
\subsection{$F_\sigma$ ideals}
Consider the ordering of all $F_\sigma$ ideals on $\omega$ denoted as $\mathfrak{F}$ ordered by reverse inclusion. The study of this ordering was initiated by C. Laflamme in \cite{LF} and also studied in \cite{JuKr}.

It is immediate from the definition that $\mathfrak{F}$ is $\sigma$-closed. To show that it is atomless, consider any $F_\sigma$ ideal $I$. Since it is not maximal there is a subset $A\subseteq \omega$ such that neither $A$ nor $\omega\setminus A$ belong to $I$. We extend $I$ by adding $A$ to obtain an ideal $I_A$; similarly, we obtain an ideal $I_{\omega\setminus A}$. They are disjoint, we must prove they are $F_\sigma$. We do it for $I_A$. Write $I$ as $\bigcup_n F_n$ where each $F_n$ is closed, thus compact. The mapping $\pi$ that sends $X$ to $X\cup A$ is continuous, thus $\pi[I]=\bigcup_n \pi[F_n]$ is still $F_\sigma$ and the downward closure of $\pi[I]$ is still $F_\sigma$ and equal to $I_A$. Since there are only $\mathfrak{c}$-many $F_\sigma$ ideals we just proved that $\mathfrak{F}$ has the BT-property. However, we do not know what the height of this ordering is.
\subsection{Summable ideals}
Consider the ordering $(c_0^+\setminus \ell^1,\leq ^*)$ where $c_0^+$ is the set of all sequences of positive reals that tend to zero and $\ell^1$ the set of all sequences of reals whose sum converges. The order relation $\leq ^*$ is almost domination, i.e. $\bar{f}\leq ^*\bar{g}$ if $\{n: g_n>f_n\}$ is finite. The investigation of this ordering was initiated by P. Vojt\' a\v s in \cite{Vo}. $(c_0^+\setminus \ell^1,\leq ^*)$ is not separative but we will show that the separative quotient is isomorphic to the set $\mathcal{I}_\Sigma$ of all summable ideals ordered by inverse inclusion.

An ideal $\mathcal{I}$ is summable if there exists $\bar{f}\in c_0^+\setminus \ell^1$ such that $\mathcal{I}=\{A:\sum_{n\in A} f(n)<\infty\}$. Note that any summable ideal is an $F_\sigma$ $P$-ideal, thus $\mathcal{I}_\Sigma$ is a subordering of $\mathfrak{F}$.

We check that $\mathcal{I}_\Sigma$ has the BT-property. Let us verify atomlessness. Let $I$ be a summable ideal determined by a sequence $(a_n)_{n=0}^{\infty}$, and let $A\in I$. Then $\sum _{i\in \omega \setminus A} a_i$ diverges; we divide $\omega \setminus A$ into two infinite subsets $B_1$ and $B_2$ such that the appropriate sums both diverge. We make new sequences $(b_n)_{n=0}^{\infty}$ and $(c_n)_{n=0}^{\infty}$ so that $b_i=a_i$ for $i\in A\cup B_1$ and $b_i=z_i$ for $i\in B_2$, where $(z_n)_{n=0}^{\infty}$ is an arbitrary converging sequence of positive reals. $(c_n)_{n=0}^{\infty}$ is defined similarly, just $B_1$ and $B_2$ change their roles. Both $(b_n)_{n=0}^{\infty}$ and $(c_n)_{n=0}^{\infty}$ diverge. We denote the appropriate summable ideals $I_b$ and $I_c$. It is clear that $I_b ,I_c\supseteq I$ and that they are disjoint.

Let $(I_j)_{j\in \omega}$ be an increasing (in inclusion) sequence of summable ideals. Let $(a^j_n)_{n=0}^{\infty}$ be the sequence of positive reals that determines the ideal $I_j$. We may assume that $(a^0_n)_{n=0}^{\infty}\geq (a^1_n)_{n=0}^{\infty}\geq \ldots$. Let $n_0$ be such that $\sum _{j\leq n_0} a_j^0>1$. We set $a_n=a_n^0$ for $n\leq n_0$. Then we find a $n_1>n_0$ such that $\sum _{j=n_0+1}^{n_1} a_j^1>1$ and set $a_n=a_n^1$ for $n_0<n\leq n_1$. And so on to obtain the whole sequence $(a_n)_{n\in \omega}$ so that $(a_n)_{n\in \omega}\leq ^* (a_n^j)_{n\in \omega}$ for all $j\in \omega$.

To verify separativness, consider ideals $I_a$ and $I_b$ corresponding to sequences $(a_n)_{n=0}^{\infty}$ and $(b_n)_{n=0}^{\infty}$, such that $I_a\nsupseteq I_b$, i.e. there is a set $B\in I_b$ which does not belong to $I_a$. That means $\sum _{k\in B} b_k <\infty$ but $\sum _{k\in B} a_k =\infty$. If $\omega \setminus B$ belongs to $I_a$ then $I_a$ and $I_b$ are already disjoint, if this is not that case then we make a new sequence $(c_n)_{n=0}^{\infty}$ such that $c_n=a_n$ for $n\in B$ and  $\sum _{k\in \omega \setminus B} c_k <\infty$. The corresponding ideal $I_c$ is below $I_a$ and disjoint with $I_b$.

It is easy to check that if $(a_n)_{n=0}^{\infty}\approx _{\mathrm{sep}}(b_n)_{n=0}^{\infty}$, i.e. $\forall (c_n)_{n=0}^{\infty}\in (c_0^+\setminus \ell^1,\leq ^*) ((c_n)_{n=0}^{\infty}\perp (a_n)_{n=0}^{\infty} \Leftrightarrow (c_n)_{n=0}^{\infty}\perp (b_n)_{n=0}^{\infty})$, then $(a_n)_{n=0}^{\infty}$ and $(b_n)_{n=0}^{\infty}$ determine the same summable ideal and the mapping $\Phi :(c_0^+\setminus \ell^1,\leq ^*) \rightarrow (\mathcal{I}_\Sigma,\supseteq )$, defined as $\Phi ((c_n)_{n=0}^{\infty})=\{A\subseteq \omega :\sum _{n\in A} c_n<\infty\}$, is an onto homomorphism of orderings preserving the disjointness relation. And the preimage of each summable ideal is precisely an equivalence class of sequences in $\approx _\mathrm{sep}$.
\begin{prop}
$\mathfrak{t}((c_0^+\setminus \ell^1,\leq ^*))=\mathfrak{t}$.
\end{prop}
\begin{proof}
To simplify the notation we will write $\bar{a}$ instead of $(a_n)_{n=0}^\infty$. Let $(\bar{a}_{\alpha})_{\alpha <\kappa}$ be a descending chain of sequences from  $(c_0^+\setminus \ell^1,\leq ^*)$ of length $\kappa <\mathfrak{t}$. We use the methods from \cite{BSc} to show it has a lower bound.

For each $\alpha <\kappa$ let $h_{\alpha}:\omega \rightarrow \omega$ be a function such that $\forall n\in \omega (\frac{1}{h_{\alpha}(n)}\leq \bar{a}_{\alpha ,n})$. Since $\kappa <\mathfrak{t}\leq\mathfrak{b}$, there is a function $h\in \omega ^{\omega}$ that almost dominates all $h_{\alpha}$'s, i.e. $h\geq ^*h_{\alpha}$ for all $\alpha <\kappa$.

Similarly, for each $\alpha <\kappa$ let $f_{\alpha}:\omega \rightarrow \omega$ be a function such that $\forall n\in \omega (\sum _{f_{\alpha}(n)\leq i<f_{\alpha}(n+1)} \bar{a}_\alpha(i)>1)$. Since $\kappa <\mathfrak{t}\leq\mathfrak{b}$, there is a function $f\in \omega ^{\omega}$ that almost dominates all $f_{\alpha}$'s, i.e. $f\geq ^*f_{\alpha}$ for all $\alpha <\kappa$. Define $g\in \omega ^{\omega}$ recursively so that $g(0)=f(0)$ and $g(n+1)=f(g(n)+1)$. Note that for every $\alpha <\kappa$ and all but finitely many $n$'s $\sum _{g(n)\leq i<g(n+1)} \bar{a}_\alpha(i)>1$ since $g(n)<f_{\alpha}(g(n))<f_{\alpha}(g(n)+1)\leq g(n+1)$. We denote $I_n$ the interval $[g(n),g(n+1))$.

For every $n$, we denote $\mathcal{F}_n$ the following set of functions $\{F: \mathrm{dom}(F)\subseteq I_n \wedge \mathrm{rng}(F)\subseteq \{\frac{1}{2|I_n|},\frac{2}{2|I_n|},\ldots ,1\}\wedge \sum _{i\in \mathrm{dom}(F)} F(i) >\frac{1}{2}\}$. Let $\mathfrak{F}=\bigcup _{n\in \omega} \mathcal{F}_n$. In the following we shall treat $\mathfrak{F}$ as $\omega$.

For every $\bar{a}_{\alpha}$, let $$X_{\alpha}=\{F: \exists n (F\in \mathcal{F}_n \wedge \forall i\in \mathrm{dom}(F) (F(i)\leq \bar{a}_\alpha(i))\}$$ An easy pigeon-hole type argument shows that it is infinite for every $\alpha <\kappa$. It is also clear that $X_{\beta}\setminus X_{\alpha}$ is finite for $\alpha <\beta$. Since $\kappa <\mathfrak{t}$, there is a lower bound $X\subseteq \mathfrak{F}$. By shrinking it if neccessary, we can assume that $|X\cap \mathcal{F}_n|\leq 1$ for every $n$. Finally, we define a sequence $\bar{a}$ as follows:

For every $m\in \omega$, if there exists $F\in X$ such that $m\in \mathrm{dom}(F)$ then we set $\bar{a}(m)=F(m)$. Otherwise, we set $\bar{a}(m)=\frac{1}{h(m)}$. It is now easy to check that $\bar{a}$ is the desired lower bound.

To prove the converse, let us at first prove that $\mathfrak{t}((c_0^+\setminus \ell^1,\leq ^*))\leq\mathfrak{b}$. Suppose the contrary. Let $(b_{\alpha})_{\alpha<\mathfrak{b}}$ be a system of almost increasing functions from $\omega ^{\omega}$ without an upper bound, $\pi :\mathbb{N}\times\mathbb{N}\rightarrow \mathbb{N}$ a bijection and $(l_n)_{n=0}^{\infty}$ a strictly decreasing sequence from $\ell ^1$ such that $l_n<\frac{1}{n}$ for every $n$. We define a descending chain of sequences from  $(c_0^+\setminus \ell^1,\leq ^*)$ $(\bar{a}_{\alpha})_{\alpha <\mathfrak{b}}$ as follows: $\bar{a}_\alpha(\pi (1,k))=l_k$ for $k\leq b_\alpha(0)$, for $l> b_\alpha(0)$ we set $\bar{a}_\alpha(\pi (1,l))=\frac{1}{l}$; generally,  $\bar{a}_\alpha(\pi (n,k))=l_k$ for $k\leq b_\alpha(n-1)$, for $l> b_\alpha(n-1)$ we set $\bar{a}_\alpha(\pi (n,l))=\frac{1}{l}$.

Let $\bar{a}$ be a lower bound for this chain. Define a function $f$ by $f(n)=\min\{k: \bar{a}(\pi (n,k))>l_k\}$. It is easy to check that $f$ almost dominates $(b_{\alpha})_{\alpha<\mathfrak{b}}$, a contradiction.

Now assume that $\mathfrak{t}<\mathfrak{t}((c_0^+\setminus \ell^1,\leq ^*))$. Let $(X_{\alpha})_{\alpha <\mathfrak{t}}\subseteq [\omega]^{\omega}$ be a descending chain without a lower bound. We define $f_{\alpha}\in \omega ^{\omega}$ for every $\alpha <\mathfrak{t}$ so that $f_{\alpha}(n)=k$ such that $|X_{\alpha}\cap [f_{\alpha}(n-1),f_{\alpha}(n))|\geq n+1$. Since $\mathfrak{t}<\mathfrak{t}((c_0^+\setminus \ell^1,\leq ^*))\leq \mathfrak{b}$,  we can again find $g\in \omega ^{\omega}$ such that for every $\alpha <\mathfrak{t}$ and for almost all $n$'s $|X_{\alpha}\cap [g(n-1),g(n))|\geq n+1$.

Define a chain $(\bar{a}_{\alpha})_{\alpha <\mathfrak{t}}$ of sequences as follows: $\bar{a}_{\alpha ,n}=\frac{1}{k}$ if $n\in X_{\alpha}\cap [g(k-1),g(k))$; if no such $k$ exists then let $\bar{a}_{\alpha ,n}=l_n$.

Finally, let $\bar{a}$ be a lower bound for this descending chain and define a lower bound $X=\{n: \bar{a}_n>l_n\}$ for the chain $(X_{\alpha})_{\alpha <\mathfrak{t}}$.
\end{proof}

\subsection{Meager and null ideals}
Next we consider the class of all meager ideals $\mathfrak{M}$ and the set of all ideals $\mathfrak{N}$ of measure zero; i.e. those ideals that are meager sets and null sets respectively in the Cantor space topology. Simultaneously, we study the set of all hereditary meager $\mathfrak{M_H}$ and hereditary null $\mathfrak{N_H}$ ideals, where an ideal $I$ is hereditary meager (null) if for every $X\in I^+$ the restriction $I\upharpoonright X=\{A\in I: A\subseteq X\}$ is meager (null) in the Cantor space $2^X$.

It is obvious they are both $\sigma$-closed. We show they are atomless, what the separative quotient for meager ideals is and that there is no dense subset in both of these orderings that has cardinality $\mathfrak{c}$. In fact, there are $2^{\mathfrak{c}}$ mutually disjoint elements in both orderings. Let us note that by $\approx_\mathrm{sep}$ we denote the ``separative equivalence" (in the ordering of meager ideals); i.e. $I\approx_\mathrm{sep} J$ if and only if for any meager ideal $K$, $K$ is disjoint with $I$ iff $K$ is disjoint with $J$.

We will use the following characterizations of meager and null ideals. 
\begin{prop}[Talagrand; see for example Theorem 4.1.2 \cite{BJ}]\label{Tal_meager}
An ideal $I$ on $\omega$ is meager if and only if there is a partition $(P_i)_{i\in \omega}$ of $\omega$ into finite sets such that $\bigcup _{i\in A} P_i\in I$ iff $A$ is finite.
\end{prop}
\begin{prop}[Bartoszy\' nski; Theorem 4.1.3 \cite{BJ}]\label{Bart_null}
An ideal $I$ on $\omega$ is null if and only if there exists an infinite system $\{\mathcal{A}_n:n\in \omega\}$ such that
\begin{enumerate}
\item each $\mathcal{A}_n$ is a finite set consisting of finite subsets of $\omega$
\item $\forall m\neq n\in \omega (\bigcup \mathcal{A}_n\cap \bigcup \mathcal{A}_m=\emptyset)$
\item $\sum _{n\in \omega} \mu \{X\subseteq \omega: \exists a\in \mathcal{A}_n (a\subseteq X)\}<\infty$
\item for every $X\in I$ $\exists^\infty n \exists a\in \mathcal{A}_n (X\cap a=\emptyset)$

\end{enumerate}
where $\mu$ is the Lebesgue measure on the Cantor space.
\end{prop}
\begin{prop}
There is a mapping $\Phi :(\mathfrak{M}, \supseteq )\rightarrow (\mathfrak{M_H}, \supseteq )$ such that $\forall I\in \mathfrak{M} (\Phi (I)\supseteq I \wedge \Phi (I)\approx _{\mathrm{sep}} I)$.
\end{prop}
Note that it follows that for any meager ideal $I$, $\Phi(I)$ is the least element in the equivalence class of $\approx_\mathrm{sep}$ containing $I$.
\begin{proof}
For a meager ideal $I$ consider the set $$\tilde{I}=\{A\subseteq \omega: I\upharpoonright A \text{ is not meager}\}$$ Let $(P_n)_{n\in \omega}$ be the partition of $\omega$ witnessing $I$ is meager (from \ref{Tal_meager}). $\tilde{I}$ is a hereditary meager ideal containing $I$. To see that it is meager check that $(P_n)_{n\in \omega}$ still works. Let $A\in \tilde{I}^+$ be arbitrary. Since $A\notin \tilde{I}$ we have $I\upharpoonright A$ is meager, so there is a partition $(Q_n)_{n\in \omega}$ of $A$ into finite sets such that $\bigcup _{i\in C} Q_i\in I$ iff $C$ is finite. If $\tilde{I}\upharpoonright A$ were not meager then there would be an infinite set $C\subseteq \omega$ such that $B=\bigcup _{i\in C} Q_i \in \tilde{I}\upharpoonright A$. $I\upharpoonright B$ would have to be nonmeager but then there would be an infinite set $D\subseteq C$ such that $\bigcup _{i\in D} Q_i \in I\upharpoonright A$, a contradiction.

It remains to prove that $\Phi (X)\approx _{\mathrm{sep}} X$ for any $X\in \mathfrak{M}$. But obviously if a meager ideal $I$ is compatible with a meager ideal $J$, then $\Phi (I)$ is compatible with $\Phi (J)$ and so also with $J$. Each equivalence class of meager ideals has its minimal element, the corresponding hereditary meager ideal.
\end{proof}
\begin{question}
We do not know whether the previous proposition also holds true for the class of null ideals. Let $I$ be a null ideal, is $\tilde{I}=\{A\subseteq \omega: I\upharpoonright A \text{ is not null}\}$ a (hereditary) null ideal?
\end{question}
It is easy to prove that $\tilde{I}=\{A\subseteq \omega: I\upharpoonright A \text{ is not null}\}$ is an ideal though. It is immediate that $\tilde{I}$ extends $I$. Let us check that it is downward closed. Let $A\in \tilde{I}$ and $B\subseteq A$. Suppose that $B\notin \tilde{I}$. Then $I\upharpoonright B$ is null. We use Proposition \ref{Bart_null} to obtain an infinite system $\{\mathcal{A}_n:n\in \omega\}$ witnessing it. It follows from that proposition that the same system would witness  that $I\upharpoonright A$ is null as well, which is a contradiction.

Let $A,B\in \tilde{I}$, we may assume they are disjoint. Then realize that $X\rightarrow (X\cap A,X\cap B)$ is a measure preserving homeomorphism from $A\cup B$ to $A\times B$ and it follows from the Fubini theorem that $I\upharpoonright A\times I\upharpoonright B$, so also $I\upharpoonright A\cup B$, is not null. Thus $\tilde{I}$ is closed under taking finite unions.
\begin{cor}
$(\mathfrak{M}, \supseteq )$ and $(\mathfrak{N}, \supseteq )$ are atomless, not separative, the separative quotient of $(\mathfrak{M}, \supseteq )$ is isomorphic to the ordering $(\mathfrak{M_H},\supseteq)$ of all hereditary meager ideals via the mapping $\Phi$.
\end{cor}
\begin{proof}
To prove they are atomless, let $I$ be an arbitrary meager ideal, let $A$ and $B$ be two infinite subsets of $\omega$ such that $A\cup B=\omega$ and neither $A$ nor $B$ is in $\Phi (I)$ ($\Phi (I)$ is meager, thus not maximal). Extend $I$ by $A$ and by $B$ to obtain two disjoint ideals $X_A$ and $X_B$ that are easily verified to be meager. The proof for null ideals is similar.

We claim they are not separative. Consider some maximal ideal $M$ on odd natural numbers and the ideal $F$ of finite sets on even numbers. Then let $I=\{A\cup B: A\in M \wedge B\in F\}$ and $J=\{A\cup B: A \text{ is an arbitrary subset of odd natural numbers }\wedge B\in F\}$ be two ideals, both easily verified to be meager and null. However, $I$ and $J$ are equivalent in the separative modification for both classes (meager and null) of ideals.\\
On the other hand, if $I$ and $J$ are two hereditary meager ideals such that $I\nsupseteq J$, then there is an infinite set $A\in J$ that is not in $I$. Let $K$ be an ideal generated by $I\cup \{\omega \setminus A \}$, it is clearly meager and $\Phi (K)\supseteq I$ is disjoint with $J$. 
 
It remains to show that $\Phi$ defines an isomorphism between the separative quotient of $(\mathfrak{M}, \supseteq )$ and an ordering $(\mathfrak{M_H}, \supseteq )$. But $\Phi$ obviously preserves the inclusion relation and the disjointness between ideals and from the previous proposition, $\Phi (I)\approx _{\mathrm{sep}} I$ for any $I\in \mathfrak{M}$, so we are done.
\end{proof}
Finally, we show there is no dense subset of these orderings with size $\mathfrak{c}$, thus preventing them to have the BT-property.
\begin{thm}
There are $2^{\mathfrak{c}}$ ideals that are both meager and null and that are mutually disjoint (in both $(\mathfrak{M},\supseteq)$ and $(\mathfrak{N},\supseteq)$).

In particular, neither $(\mathfrak{M},\supseteq)$ nor $(\mathfrak{N},\supseteq)$ has the BT-property.
\end{thm}
\begin{proof}
Let $(I_n)_{n\in \omega}$ be a partition of $\omega$ into intervals such that $|I_n|=n+1$. For any $X\subseteq \omega$ let $X^I$ be the set $$\bigcup _{m\in X} \{k\in \omega :\exists n (k\text{ is the }m\text{-th element of }I_n\}$$ It is clear that if $J$ is an ideal on $\omega$ then $\{X^I:X\in J\}$ is a base of an ideal; we shall denote this ideal as $\mathcal{I}_J$.

Now let $\mathcal{M}$ be the set of all maximal ideals on $\omega$, its size is $2^{\mathfrak{c}}$. For any $J\in \mathcal{M}$, we make an ideal $\mathcal{I}_J$ and obtain a system $\mathfrak{I}$ of $2^{\mathfrak{c}}$ ideals. The disjointness of two ideals $J_1, J_2\in \mathcal{M}$ is easily seen to be preserved for $\mathcal{I}_{J_1}$ and $\mathcal{I}_{J_2}$.
\\
\\
{\bf Claim 1} $\mathfrak{I}$ is a system of meager ideals.\\

The interval partition $(I_n)_{n\in \omega}$ works for all ideals from $\mathfrak{I}$. Assume that some set set $X\in \mathcal{I}_J$, where $\mathcal{I}_J \in \mathfrak{I}$, contains a union of infinitely many intervals. It is easy to check that once it contains the whole interval $I_n$ then it contains all previous intervals. Thus, we conclude that $X=\omega$ which is a contradiction.
\\
\\
{\bf Claim 2} $\mathfrak{I}$ is a system of null ideals.\\

We use characterization of null ideals from \ref{Bart_null}. For $m\leq n$, let $i^m_n$ be the $m$-th element of $I_n$. Let $A_n=\{a_n=\{i^n_m: n\leq m\leq 2n-1\}\}$. These sets satisfy the first three conditions from \ref{Bart_null}. Let $X\in \mathcal{I}_J$, where $\mathcal{I}_J \in \mathfrak{I}$, be a given set. $X\subseteq Y^I$ for some $Y\in J$. Then it is easy to check that $A_n\cap X=\emptyset$ for $n\in \omega \setminus Y$ and $|\omega\setminus Y|=\omega$; thus we also verified the last fourth condition and proved that every $\mathcal{I}_J$ is null.
\end{proof}

\section{Products}
If $P$ and $Q$ are two orderings with the BT-property then their product $P\times Q$ has again the BT-property and the height is less or equal to the minimum of heights of the original orderings. To see this, just realize that if $B$ is a regular subalgebra of $C$ then $\mathfrak{h}(B)\geq \mathfrak{h}(C)$, and that $\mathrm{RO}(P)$ is a regular subalgebra of $\mathrm{RO}(P\times Q)$. The same holds for countable products and iterations. Let us mention the case of products of $([\omega]^\omega,\subseteq^*)$. By $\mathfrak{h}_\alpha$, for $2\leq \alpha\leq \omega$ we denote $\mathfrak{h}(([\omega]^\omega,\subseteq^*)^\alpha)$. It immediately follows that $\mathfrak{t}(([\omega]^\omega,\subseteq^*)^\alpha)=\mathfrak{t}$ for any $2\leq\alpha\leq \omega$. Shelah and Spinas in \cite{ShSp} proved that consistently for any $n\in \omega$ $\mathfrak{h}_{n+1}<\mathfrak{h}_n$. However, the following question remains open.
\begin{question}
Does it hold in ZFC that $\mathfrak{h}_\omega=\min\{\mathfrak{h}_n:n\in \omega\}$?
\end{question}

\end{document}